\documentclass[12pt]{article}
\title{Random Schr\"odinger operators on long boxes, \\noise explosion and the GOE}
\date{\today }
\author{Benedek Valk\'o
\footnote{Department of Mathematics, University of
Wisconsin - Madison, WI 53705, USA. valko@math.wisc.edu}
 \hskip 3em B\'alint Vir\'ag\footnote{Department of Mathematics, University of Toronto, ON M5S
2E4, Canada. balint@math.toronto.edu} }

\usepackage{amsfonts}
\usepackage{graphicx}
\usepackage{amsmath}
\usepackage{amsthm,color}

\theoremstyle{plain}
    \newtheorem{theorem}{Theorem}
    \newtheorem{lemma}[theorem]{Lemma}
    \newtheorem{proposition}[theorem]{Proposition}
    \newtheorem{corollary}[theorem]{Corollary}

\theoremstyle{definition} % For roman text in the body
    
    \newtheorem{remark}[theorem]{Remark}

\theoremstyle{remark} % For an italic header, more subtle than definition style

\newcommand\mnote[1]{} %off
\newcommand\be{\begin{equation}}

\newcommand\ee{\end{equation}}

\newcommand{\comment}[1]{}
\newcommand{\eps}{\varepsilon}
\newcommand{\Z}{{\mathbb Z}}

\newcommand{\R}{{\mathbb R}}
\newcommand{\CC}{{\mathbb C}}

\newcommand{\ev}{\mbox{\bf E}}

\newcommand{\one}{{\mathbf 1}}

\newcommand{\dist}{\mbox{\rm dist}}

\newcommand{\sm}{{\raise0.3ex\hbox{$\scriptstyle \setminus$}}}

\newcommand{\tr}{\operatorname{tr}}

\newcommand{\ed}{\stackrel{d}{=}}

\newcommand\vect[2]{
\left[\begin{array}{c}
#1\\
#2
\end{array}\right]}

\newcommand\mat[4]{
\left[\begin{array}{cc}
#1& #2\\
#3 & #4
\end{array}\right]}
\newcommand{\diag}{\operatorname{diag}}
\newcommand{\cha}{\operatorname{cha}}
\newcommand{\spec}{\operatorname{spec}}
\newcommand{\zeros}{\operatorname{zeros}}

\usepackage{natbib}
\begin{document}
\maketitle
\begin{abstract}
It is conjectured that the eigenvalues of random
Schr\"odinger operators at the localization transition in
dimensions $d\ge 2$ behave like the eigenvalues of the
Gaussian Orthogonal Ensemble (GOE). We show that  there are sequences of $n\times m$ boxes with $1\ll m\ll n$ so that
the eigenvalues in low disorder converge to Sine$_1$, the
limiting eigenvalue process of the GOE. For the GOE case, this is the first example where Wigner's famous prediction is proven rigorously: we exhibit a complex system whose eigenvalues behave like those of random matrices.
\end{abstract}

\section{Introduction}

When \cite{Wigner57} introduced random matrices to
model large atomic nuclei, his main goal was to find a
simple mathematical model that shows repulsion between
eigenvalues as observed in the data. The Gaussian
orthogonal ensemble (GOE) has since been a remarkable
success in the physics literature: several important
complex systems are predicted to have eigenvalue repulsion
akin to the GOE, most notably the Laplacians of many planar
domains, see \cite{bohigas}.

Given  all the non-rigorous theories that predict GOE
behavior, as well as a lot of numerical evidence, it is
surprising that there are hardly any mathematically
rigorous results in this direction. Most objects with
rigorously known bulk GOE behavior are themselves random
matrices constructed similarly to the Gaussian orthogonal
ensemble, see \cite{TV} and \cite{EY2} for recent
breakthrough results in this direction.

The goal of this paper is to provide the first example where Wigner's  prediction is proven rigorously: a complex system whose eigenvalues behave like those of random matrices. To clarify, the Gaudin-Mehta theorem implies that the rescaled eigenvalue process of the $n\times n$ GOE matrix converges to a limiting point process, which we call Sine$_\beta$, $\beta=1$; the same holds for GUE and GSE matrices, with $\beta=2,4$, respectively (see \cite{BVBV} for general $\beta$). A model is called (bulk) {\bf universal} if its eigenvalue process converges to Sine$_\beta$, for  $\beta=1,2$ or $4$.

We show that certain random Schr\"odinger operators on long $m\times n$ boxes, $1\ll m\ll n$ behave like the Gaussian orthogonal ensemble. More precisely, we will show

\begin{theorem}\label{thmmain}
There exists a sequence of $m\times n$ boxes, $1\ll m\ll n$  so that
the rescaled eigenvalue process of the adjacency matrix
plus suitable diagonal noise converges to the Sine$_1$
point process.
\end{theorem}

Theorem \ref{thmmain} will be proved in Corollary
 \ref{Sine1 limit} which is based on  Proposition \ref{shorttime} (see Section \ref{section:GOE}).

The boxes we consider are very long, so in some ways our result is related to the quasi-one dimensional setting,
except that both dimensions of the box converge to infinity.

Recall that the Sine$_1$ point process is the limit of the
eigenvalues at the bulk of the spectrum of the Gaussian
orthogonal ensemble, see \cite{mehta}. The process-level
convergence described here implies convergence of
eigenvalue gap sizes and all similar local statistics.

\subsubsection*{Lack of true universality in one dimension}

We would like to put the present work in the context of two results. The first is due to \cite{KS}. They
consider the CMV matrices with random Verblunski coefficients, which are a unitary analogue of a random
Schr\"odinger operators. Their focus is the CMV representation of the Haar unitary random matrix.  The entries of this matrix are built out of the Verblunski coefficients, which, in this case are independent, but non-identically distributed complex random variables with a spherically symmetric distributions. Naturally, since these are conjugates of Haar unitary matrices, the eigenvalue process
converges to the  Sine$_2$ process, the bulk limit of the GUE. Moreover, \cite{KS} show that this holds even if the distribution of the random variables is changed, as
long as spherical symmetry and the variances are kept. In this sense, this is a universality result. (A similar result, for the Hermitian case was proven by \cite{BVBV}, but there, due to the lack of spherical symmetry, more moments need to be fixed.)

However, in one-dimension, a very special, non-iid variance structure was necessary to get the Sine$_2$ limit. In fact \cite{KS} show that a constant factor change for
all variables will change the limiting process (they call the limit $C\beta E$, later shown to be identical with Sine$_\beta$). Many limiting processes, can be achieved by adjusting the variances judiciously. It is expected that the more natural model, in which the variances are identical, would give a version of the Sch$_\tau$ process, described below.

The standard critical one-dimensional Schr\"odinger operators were treated in \cite{KVV}: we considered adding i.i.d.~diagonal noise
with variance $\sigma^2/n$ to the adjacency matrix of a 1-dimensional path of length $n$. In this critical regime, as
$n\to\infty$ we got a limiting  point process Sch$_\tau$, whose parameter is given by $\tau=\sigma^2/(2-\lambda_*^2)$, where $\lambda_*$ is the macroscopic location of our window in
the spectrum. The key is that Sch$_\tau$ is not a limit of GOE, GUE, or other usual random matrix models. It is a different point process with stronger
repulsion and long-range order. As $\tau$ ranges from $0$ to $\infty$, it interpolates between the clock and Poisson distributions. In \cite{KVV} it is also shown that for a special sequence of decaying variances, all Sine$_\beta$ processes can also be achieved as limits. Again, this is not true universality: GOE-like limits appear only for a careful choice of variances.

\subsubsection*{Universality when $1\ll m\ll n$}

As we have seen, in order to find truly universal GOE/GUE phenomenon in Schr\"odinger operators, one has to go beyond one dimension. This is what we do here,
by considering an $m\times n$ box, with $1\ll m \ll n$ with diagonal noise added.

Our strategy is similar to that of \cite{KatzSarnak}, who proved the first true universality result
for GUE for zeta functions of random algebraic curves over finite fields. There are two parameters, the degree $n$ and characteristic $p$. They first let
$p\to \infty$ and use the Deligne equidistribution theorem to show that the zeros of that zeta function have the same distribution as eigenvalues of a Haar unitary matrix. They then conclude universality from a double limit argument (Section 13.8).

We proceed similarly. We first find a regime where as $n\to\infty$ the eigenvalues of our operator converge to that of a certain $m\times m$ random matrix which is similar, but not identical, to the GOE. Then, using results of Erd\H os, Yau, et al and a diagonalization argument we show that for a certain sequence of $1\ll m\ll n$, we get the universal Sine$_1$ limit.

\subsubsection*{Random Schr\"odinger operators}
The concept of localization for random Schr\"odinger
operators was introduced by \cite{Anderson}. As more and
more diagonal noise is added to the Laplacian of a large
box, the eigenvectors change from being spread out over the
entire box to localized on smaller regions. It is believed
(see, for example, \cite{Altshuler}, \cite{efetov}) that
near this transition the eigenvalue process is GOE-like,
and, as the noise increases, it is approximately Poisson.

In the long box case, we have a related phenomenon. As $n\to\infty$,
we get a limiting point process $\Lambda_{m,\sigma}$ depending on the noise. In a certain
scaling, when $\sigma\to 0$, this $\Lambda_{m,\sigma}$ converges to the eigenvalue process of an $m\times m$ Gaussian random matrix.
More precisely, we have

\begin{theorem}[Eigenvalue process transition]\label{thm2}  Consider boxes in $\mathbb
Z^2$ with fixed base $\Z_m$ and fix a weight $0<r\le 1$. Consider the process of eigenvalues of
 \begin{align}\label{scaled spec}
 n(({r}\Z_m) \times \mathbb
Z_n+\frac{\sigma}{\sqrt{n}}V-I\lambda^*),
\end{align}
 where $V$ are
diagonal matrices with independent, mean zero, variance 1
entries with bounded third moment. Then
\begin{itemize}
\item For almost all  $\lambda^*\in (-2(1-r\cos(\pi/(m+1))),2(1-r\cos(\pi/(m+1)))$, for appropriate subsequences, this process has a limit
$\Lambda_{\sigma}$ which depends on $m, r$ and $\lambda^*$.
\item{} As $\sigma\to 0$, the  process ${\sigma^{-1}}\Lambda_\sigma$
converges to the  (randomly shifted) eigenvalue process of
a random matrix with independent {real} Gaussian entries.
\item For every $\sigma$, the process $\Lambda_\sigma$ is
the zero process of the determinant of a $m$-dimensional
matrix-valued analytic function described by a stochastic
differential equation.
\end{itemize}
\end{theorem}
\begin{remark}
To illustrate the transition to Poisson, it is possible
to show that
\begin{itemize}
\item
{\it as $\sigma\to\infty$ the process $\Lambda_\sigma$
converges to the Poisson point process.}
\end{itemize}
We plan to do this in a subsequent article.
\end{remark}

The proofs of both theorems are based on the analysis of
transfer matrices. The proof of Theorem \ref{thm2} is given in Section \ref{section:limit}, it is based on Proposition \ref{letm} and Corollary \ref{eigenvalues
converge} which are proved in the same section.
%Theorem \ref{thmmain} will follow from  Proposition \ref{shorttime} and Corollary
% \ref{Sine1 limit} (see Section \ref{section:GOE}).

The noise term in the transfer matrix evolution in this scaling regime for block Jacobi matrices has been studied in \cite{rs} using a language different from SDEs.
The first arxiv version of the present paper was followed by the preprint of the paper \cite{BD}, who, in independent work, also study SDE limits
of transfer matrices. Their starting point is the so-called
DMPK theory in the physics literature (due to Dorokhov-Mello-Pereyra-
Kumar), which is essentially
the study of diffusive limits of quasi-one-dimensional
random Schr\"odinger operators from a slightly different
point of view. We refer the reader to \cite{BD} for a
discussion of this theory. One of the novelties of our
approach is that it allows for studying the dependence on
the eigenvalue $\lambda$, which in turn allows us to deduce
the scaling limit of the spectrum, the main focus here.

\subsubsection*{Noise explosion} The central reason for the
appearance of random matrices is noise explosion. For a
simple case of this, consider
%$$
%f_n(t)=\frac{1}{n}\sum_{k=0}^{\lfloor nt\rfloor} e^{i\eta
%k}
%$$
%then for $\eta=0$ mod $(2\pi)$ and large $n$ the function
%$B_n(t)$ approximates $t$, but for $\eta \not=0$ it is
%approximately $0$. However, if we add independent noise
%terms $X_i$, say with standard normal distribution, then
$$
B_{\eta,n}(t)=\frac{1}{\sqrt{n}}\sum_{k=0}^{\lfloor
nt\rfloor} e^{i\eta k}X_k,
$$
where the $X_k$ are independent random variables, with,
say, standard normal distribution.  If $\eta=0\;($mod
$2\pi)$, then the sequence of functions $B_{\eta,n}$
converges in law to real Brownian motion. Otherwise,
$B_{\eta,n}$ converges to complex Brownian motion. Thus a
one-dimensional noise process $X_k$ gives rise to a two
real-dimensional Brownian motion in the presence of
oscillatory terms. Moreover, if $\eta,\nu$ are linearly
independent over the integers, then $B_{\eta,n},$
$B_{\nu,n}$ converge jointly to two independent Brownian
motions.

Now we change the setting so that the $X_k$ are $m\times m$
{diagonal} matrices with independent standard normal
entries. Let $U$ be a unitary matrix with eigenangles that
are linearly independent over the integers, and also assume that the absolute squares of the eigenvectors of $U$ are not orthogonal to each other.
Now consider the sum
$$
B_n(t)=\frac{1}{\sqrt{n}}\sum_{k=0}^{\lfloor nt\rfloor}
U^{-k}X_kU^k.
$$
A simple computation shows that $B_n(t)$ converges to a Hermitian matrix-valued
Brownian motion process which is $m^2$-real dimensional. In
such oscillatory sums, noise that was originally restricted
to the diagonal explodes into all possible directions, and
changes dimension from $m$ to $m^2$. This phenomenon, which
we call {\bf noise explosion}, plays a central role in the
proofs below.
\bigskip

The method of analysis in this paper is an extension of the
1-dimensional case introduced in a different setting in
\cite{BVBV} and further refined and simplified in
\citet*{KVV}. The latter paper studies the most natural
1-dimensional random Schr\"odinger operator at the
localization transition.

\subsubsection*{Open problems}

The current paper does not give any bounds on the aspect ratio of the long boxes $m/n$. This would require rates of convergence estimates for diffusion approximation.

{\bf Problem 1.} Show that the result in this paper holds for $n=O(m^q)$ for some  $q\ge 1$. What is the smallest possible value of $q$ for which the results hold? Is it $q=1$?

\smallskip

{\bf Problem 2.} Extend the results to long boxes of higher dimension. The eigenvalue structure of the base here is more complicated, see Remark \ref{r:hd}.

\smallskip

Our results apply for weighted boxes and the larger the weight is, the stronger the restrictions are on the centering $\lambda_*$. In particular, when $r=1$ we will need to take a sequence of $\lambda_*\to 0$ to get the Sine$_1$ process in the limit (see Corollary  \ref{Sine1 limit}). However, the results should hold on a dense subset of $[-2,2]$.

{\bf Problem 3.} Show that  for any fixed weight $r$ the results hold for almost every $\lambda_*\in (-2,2)$.

% Eliminate the $r\to 0$ condition in Proposition \ref{shorttime}. \benedek{or Corollary \ref{Sine1 limit}? } This would require more general versions of the results in \cite{EYY}.

\section{Description of the model and notation}\label{section:model}

Let $\mathbb Z_n$ denote the graph of a path of length $n$
with vertices $1,2,\dots,n$. We will use the same variables
for a (weighted) graph and its adjacency matrix, so
$\mathbb Z_n$ will also denote the $n\times n$ matrix with
entries $(\mathbb Z_n)_{i,j}=1_{\{|i-j|=1\}}$. Let $\times$
denote the product of weighted graphs (or
matrices), namely
$$
(A \times B)_{(i,i'),(j,j')} =
\one_{i'=j'}A_{i,j}+\one_{i=j}B_{i',j'}
$$
With tensor product, it can be expressed as  $A\otimes I + I\otimes B$. For a positive real $r$ and weighted graph (or matrix) $G$
we will denote by $rG$ the graph where the weights are all
multiplied by $r$ (which is the same as multiplying the
matrix $G$ by $r$).

{\noindent \bf {Weighted boxes.}} The boxes we consider are $\left[r \Z_m \right]\times \Z_n$
 for some positive
real $r$, where $n$ is typically much larger than $m$. This is the adjacency matrix of a weighted graph on
the  box $\Z_m \times \Z_n$ where the edges in
the first  direction have weight $r$, and the edges
in the second direction have weight $1$.

The adjacency matrix corresponds to Dirichlet boundary
conditions. For a probabilistic interpretation, the matrix
$$r\mathbb Z_m\times \mathbb Z_n-2(r+1)I_{m\times n }$$ is the transition
probability matrix of the continuous time random walk on
the box killed when it leaves; the jump rates are $r$ and
$1$ in the two different kinds of directions.

%\bigskip
%
%{\bf \noindent General slabs.} Some of our results will
%apply to  $rG\times \Z_n$ where $G$ is a general symmetric
%matrix.

Some of our results will
apply to  $G\times \Z_n$ where $G$ is a general symmetric
matrix with nonnegative entries.

\bigskip

{\noindent \bf Chaoticity.} For the noise explosion
phenomenon to work, we need to consider conjugation by
special angles. Consider points (angles) $\mathcal
X=\{x_1,\ldots, x_m\}$ on the unit circle $\mathbb R/2\pi$.
We will be interested in the set
\[
\mathcal{A}:=\{x_{i_1}+x_{i_2}+x_{i_3}+x_{i_4}\} \cup
\{x_{i_1}+x_{i_2}-x_{j_1}-x_{j_2}:  i_k \neq j_\ell\} \cup \{x_{i_1}+x_{i_2}+x_{i_3}-x_{j_1}:i_k\neq j_1\}
\]
where the indices run through 1,\dots, $m$ and the
operations are meant mod $2\pi$. The {\bf chaoticity} of
$\mathcal X$ is defined as the distance of the set
$\mathcal A$ from 0
\begin{equation}\label{def:cha}
\cha(\mathcal X)= \dist(0,\mathcal A)
\end{equation}
We call the $\mathcal X$ {\bf chaotic} if it has nonzero
chaoticity.

 In particular, if
the $\mathcal X$ is chaotic then then the  $x_i$'s  are
distinct. Moreover, if the $x_i$'s are linearly independent
over the integers then they are chaotic. A condition very similar to chaoticity appears as the main assumption in \cite{HSB}.

%\bigskip

%The following property is shared by boxes in $\mathbb Z^d$
%and will be used in the proof.

\bigskip

The \textbf{noise explosion} described in the introduction will rely on the following quantities. Let $G=OD{O^{-1}}$ be the
diagonalization of $G$, so that $O$ is orthogonal and $D$
is diagonal. Let $O_i=O_{i,\cdot}$ denote the $i$th row of
$O$. Let $|O_i|^2$ denote the vector formed by the squares
of the entries of $O_i$.

When $G$ is $\Z_m$ (i.e.~its  adjacency matrix), then the eigenvalues are given by $d_j=2\cos(\pi j/(m+1))$.
The eigenvectors are of the
form
$$O_{jk}=\sqrt{2/(m+1)}\sin(\pi jk/(m+1)),$$ and it is easy
to check that
\begin{equation}(m+1)\langle |O_i|^2, |O_j|^2 \rangle=
\begin{cases}
3/2 &\text{for~}i=j \,,\\
1 &\text{for~}i\not=j\,.\\
\end{cases}\label{isotropic}
\end{equation}

\section{The regularized transfer matrix evolution}\label{section:reg}

For this section, consider the operator $G\times \mathbb
Z_n$ where $G$ is a symmetric matrix of dimension $m\times
m$.

Suppose that we want to solve the equation $Mu=\lambda u$
where $M=G\times \mathbb
Z_n+V$ with a diagonal matrix $V$. We will do this
by solving the system of linear equations recursively, slice
by slice (where a slice is a copy of $G$, i.e. $G\times
\{i\}$). Let $u$ be a function from the vertices of $G
\times \Z_n  $ to $\mathbb R$, and let $u_i(j) = u(j, i)$
so that $u_i$ is a vector indexed by $G$. Denote by $V_i$
the diagonal matrix which is the restriction of $V$ to the
indices $G\times \{i\}$. Then the eigenvalue equation for
entries with the same first coordinate $2\le i\le n-1$
reads
\begin{equation}\label{ev equation}
\lambda u_i = (G +V_i)u_i + u_{i-1}+ u_{i+1}
\end{equation}
and if we set $u_0\equiv u_{n+1}\equiv 0$, then \eqref{ev
equation} holds for $i=1$ and $n$ as well. It is not hard
to check that this system of $n$ equations is then
equivalent to $Mu=\lambda u$. Note that \eqref{ev equation}
can be written in terms of block transfer matrices of
dimension $2m$ as follows:
$$
\vect{u_{k+1}}{u_{k}} =  \mat{\lambda I -G-V_k}{-I}{I}{0}
\vect{u_{k}}{u_{k-1}}.
$$
Denoting the above matrix by $T_k$, we can now characterize
the eigenvalues of $M$ as
\begin{eqnarray}
\{\lambda\;:\;\exists u_1,u_n \textup{ such that
}\;T_n\cdots T_2 T_1\vect{u_1}{0} \| \vect{0}{u_n}\}
=\{\lambda\;:\;\det\left((T_n\cdots T_2
T_1)_{(1,1)}\right)=0\}\label{EVchar}.
\end{eqnarray}
Here the subscript $(1,1)$ refers to the top left $m\times
m$ submatrix, and both representations are equivalent to
this submatrix having some eigenvector $u_1$ with
eigenvalue 0. So we may write
 \begin{eqnarray}
\spec(M)  &=&\zeros_\lambda \left[\det\left((T_n\cdots T_2
T_1)_{(1,1)}\right)\right]\label{specM}
 \end{eqnarray}
where $\spec(\cdot)$ will refer to the eigenvalue counting
measure and $\zeros_\lambda $ refers to the zero counting
measure of a real-analytic function of $\lambda$. It is not
hard to check that the equality also holds in the presence
of multiple zeros.

We now introduce the unperturbed version of the transfer
matrix
$$ T_{*}=\mat{\lambda_* I
-G}{-I}{I}{0}.
$$
where $\lambda_*$ is some fixed reference point. Let $G= OD
O^T$ be the diagonalization of $G$ with $O$ orthogonal and
$D=\diag(d_1,\dots, d_m)$ diagonal. We will use the
shorthanded notation $X^Y=Y^{-1} X Y$ which means
$G=D^{O^T}$.

We may change basis to study  products of
$$
T_k^{O\otimes I_2} =\mat{\lambda I - D-V_k^{O}}{-I}{I}{0},
\qquad O\otimes I_2={\mat{O}00{O}}.
$$
and note that
$$ T_{*}^{O\otimes I_2}=\mat{\lambda_* I
-D}{-I}{I}{0}
$$
This matrix can be completely diagonalized as it is just
the permutation of a block diagonal matrix with $2\times 2$
blocks. (We have learned that such a
change of basis has been considered for a slightly
different problem by \cite{HSB}.) Let
\begin{align}\label{defZ}
z_j=(\lambda_*-d_j)/{2}+i \sqrt{1-\left((\lambda_*-d_j)/
{2}\right)^2}\quad \textup{and} \quad Z=\diag(z_1, \dots, z_m).
\end{align}
%Let $Z$ be a
%diagonal matrix so that $z_j=Z_{jj}$ is the point on the
%unit circle in the upper half plane with real part
%$(\lambda_*-d_j)/2$, or in short
%$$
%Z = (\lambda_*I-D)/2+ i\sqrt{1-((\lambda_*I-D)/2)^2}.
%$$
We assume that $|\lambda_*-d_j|\in (0,2)$ for all $j$ which
means that $z_j$ is a unit length complex number in the
upper half plane.

 Then we
have
$$ T_*^{(O\otimes I_2)Q}=\mat{\bar Z
}00{Z}.
$$
where
\begin{equation}\label{defS}
S=\mat{\diag|\bar z_j-z_j|^{-1/2}}{0}{0}{\diag|\bar
z_j-z_j|^{-1/2}}, \qquad Q=\mat{\bar Z }{Z}{I}{I} S
\end{equation}
and
$$
Q^{-1}=iS\mat{{\phantom-}I}{-Z}{-I}{{\phantom-}\bar Z}
$$
Then the matrix
\begin{align}\label{Xdef}
X_k = (T_*^{-k} T_k\cdots T_2T_1)^{(O\otimes I_2)Q},
\end{align}
is deterministic matrix factors away from the transfer
matrix evolution $T_k\cdots T_1$. However, it has the
advantage that, unlike the product $T_k\cdots T_1$, it
changes slowly as $k$ varies. This can be seen from its
evolution. From the definition we have
\begin{equation}
X_k=(\;T_*^{-k}\;T_{k}\;T_*^{-1} \;T_*^{k}\;)^{(O\otimes
I_2)Q}\; X_{k-1},\label{X:evo}.
\end{equation}
Note that in our case $T_k$ will be a small perturbation of
$T_*$.

We call $(X_k,0\le k \le n)$ the {\bf regularized transfer
matrix evolution}, regularized at $\lambda_*$. In the next
section, we will show that this evolution has a stochastic
differential equation limit. But first let us check how to
read off the eigenvalues of the operator $M$ from $X_n$. By
(\ref{specM}) the eigenvalues of $M$ are given by
\begin{eqnarray*}
 &&\left\{\lambda\;:\;
\det\left(((T_*^{(O\otimes I_2)Q})^n X_n)^{Q^{-1}(O^{-1}\otimes I_2)})_{(1,1)}\right)=0\right\}\\
\qquad
&&\qquad\qquad\qquad= \left\{\lambda\;:\; \det\left(((T_*^{(O\otimes
I_2)Q})^n X_n)^{Q^{-1}})_{(1,1)}\right)=0\right\}
\end{eqnarray*}
Note that $(T_*^{(O\otimes I_2)Q})^nX_n$ is a matrix of the
form
$$
\mat{\bar Z^n X_{11}}{\bar Z^n X_{12}}{Z^n X_{21}}{Z^n
{X_{22}}},
$$
and so the $(1,1)$ block of $Q(T_*^{(O\otimes
I_2)Q})^nX_nQ^{-1}$ is given by
$$
iS_{11}(\bar Z^{n+1} X_{11}  -\bar Z^{n+1} X_{12}  +
Z^{n+1} X_{21} - Z^{n+1} X_{22})S_{11} =iS_{11}(\bar
Z^{n+1},Z^{n+1})X_n(I,-I)^TS_{11}
$$
Since $S_{11}$ is nonsingular, have
\begin{eqnarray}
\spec(M)&=&
  \zeros_\lambda \left[\det ((\bar
Z^{n+1},Z^{n+1})X_n(I,-I)^T)\right]\nonumber
 \\&=& \zeros_\lambda \left[\det \Im(\bar
Z^{n+1}((X_n)_{11}-(X_n)_{12}))\right].\label{specM:X}
\end{eqnarray}
the second equality follows from the fact that the
determinant is zero only for real $\lambda$ (as these are
eigenvalues of a symmetric matrix $M$), and for real
$\lambda$, we have $X_{22}=\bar X_{11}$ and $X_{21}=\bar
X_{12}$. This will be shown in the proof of Proposition
\ref{letm} in the next section.

%We will analyze the limit of this process in the next section.

\section{The limiting  transfer matrix evolution}
\label{section:limit}

The goal of this section is to show that the regularized
transfer matrix evolution introduced  in the previous
section has a stochastic differential equation limit. It
will follow that the eigenvalues of the corresponding
operator converge in distribution to an explicitly
constructible limit.

We first state the scaling limit of the evolution of
transfer matrices.

\begin{proposition}[Limiting evolution of transfer
matrices]\label{letm} Fix $r, \sigma>0$. %, and let $n\to\infty$.
Consider the operator $M_n=rG \times Z_n + n^{-1/2} \sigma
V$, where $V=V^{(n)}$ is a diagonal matrix with independent
random entries of mean $0$, variance $1$ and uniformly
bounded third moments. Denote the eigenvalues of $G$ by
$d_1, \dots, d_m$ and assume that $0<|\lambda_*-{r}d_j|<2$ for
all $j$ and that the angles $ \arccos((\lambda_*-r
d_j)/2)$, $j=1,\dots, m$ are chaotic. Consider the
regularized transfer matrix evolution (\ref{Xdef}) corresponding to
$\lambda_*$: $(X^{(n)}_k,0\le k\le n)$.

Then for any finite  $\Lambda \subset \mathbb C$ we have
convergence in distribution as $n\to \infty$:
$$
(X^{(n)}_{\lfloor
nt\rfloor}(\lambda^*+\lambda/n),t\in[0,1],\lambda\in
\Lambda)\Rightarrow (Y_t(\lambda),t\in[0,1],\lambda\in
\Lambda)
$$
where $Y=Y_t(\lambda)$ is the {strong} solution of the SDE
\begin{equation}\label{mainsde}
dY = S^2 \mat I00{-I} i \lambda dt Y + i\sigma S\mat
{dA}{dB}{-d\bar B }{-d\bar A}SY, \qquad Y_0=I.
\end{equation}
and $S$ is defined in (\ref{defS}).
Moreover, for any fixed $t$, the functions $\lambda\to X^{(n)}_{\lfloor
nt\rfloor}(\lambda^*+\lambda/n)$ converge in distribution to the random analytic function $Y_t(\cdot)$
with respect to the uniform-compact
topology of functions.
%
%we have convergence in
%distribution of \benedek{$X^{(n)}_{\lfloor nt\rfloor}(\lambda^*+\frac{1}{n}\cdot)
%\rightarrow Y_t(\cdot)$} with respect to the uniform-compact
%topology of functions.
The covariance structure of the
matrix-valued Brownian motions $A(t), B(t)$ is as follows.
We have
\begin{equation}\label{covariances1}
A_{ij}=\bar A_{ji},\qquad B_{ij}=B_{ji}
\end{equation}
($B$ is complex symmetric, not Hermitian) and
\begin{equation}\label{covariances2}
\ev |B_{ij}|^2 = \ev A_{ii}A_{jj}=\ev A_{ii}\bar A_{jj}=\ev
|A_{ij}|^2 =\langle |O_i|^2, |O_j|^2  \rangle t
\end{equation}
(where $O_i$ are defined in Section \ref{section:model}) and all covariances that do not follow from the above are
zero.
\end{proposition}

\begin{remark}\label{remG}
In the  case $G=\Z_m$  by (\ref{isotropic}) we have
\begin{equation*}
A(1)=\frac{1}{\sqrt{m+1}}(A'+\zeta I), \qquad
B(1)=\frac{1}{\sqrt{m+1}} B'
\end{equation*}
Here $A'$ is a version of the GUE with $A_{i,j}'=\bar
A_{j,i}'$ standard complex normals for $i\neq j$ and mean
zero variance 1/2 i.i.d real normals in the diagonal. (The
usual GUE would have variance 1 in the diagonal.) The
random variable~$\zeta$ is standard normal and independent
of $A'$.

 $B'$ is an independent  symmetric matrix with i.i.d.~standard complex normal
entries above and below the diagonal and i.i.d.~mean zero
complex normals with variance 3/2 in the diagonal (which
are independent of everybody else).
\end{remark}

%\benedek{Should we include a remark on the $m=1$ case and KVV?}

Next, we consider the eigenvalues of the operators $M_n$.

\begin{corollary}\label{eigenvalues
converge} Consider the operators $M_n$ of Proposition
\ref{letm}. Assume that along some subsequence $\bar
Z^{n+1}\to Z_*$ as $n\to\infty$ where $Z$ is defined in (\ref{defZ}) and $Z_*$ is a fixed $m\times m$ matrix. Then on this
subsequence we have
$$
\spec(n(M_n-\lambda^* I))\Rightarrow \zeros_\lambda \left[\det ((\bar
Z_*,Z_*)Y_1(\lambda)(I,-I)^T)\right].
$$
\end{corollary}

We proceed with the proof of Proposition \ref{letm}. The
proof relies on the noise explosion phenomenon introduced
in Section 1.

\begin{proof}[Proof of Proposition \ref{letm}]
From now on, with a slight abuse of notation, we will use $X^{(n)}_k(\lambda)$ for $X_k^{(n)}(\lambda^*+\lambda/n)$ and sometimes we will drop the dependence on $\lambda$ and/or $n$.

%We will assume $\sigma=1$, the general case may be proved the same way.
We first study the evolution without scaling time. From
(\ref{X:evo})
%$$
%X_{k} =(\;T_*^{-k}\;T_{k}\;T_*^{-1} \;T_*^{k}\;)^{(O\otimes I_2)Q}\;
%X_{k-1}
%$$
%and therefore
$$
X_{k}-X_{k-1}  = ( T_*^{(O\otimes
I_2)Q})^{-k}(T_kT_*^{-1}-I)^{(O\otimes
I_2)Q}(T_*^{(O\otimes I_2)Q})^k \; X_{k-1}
$$
The coefficient of $X_{k-1}$ above is given by
\begin{equation}\label{evoR}
R_k=\mat{Z^k}00{Z^{-k}}Q^{-1}\mat
0{\frac{\lambda}{n} I+\frac{\sigma}{\sqrt{n}}V_k^O}00Q\mat{Z^{-k}}00{Z^{k}}.
\end{equation}

Note that the
vector $(X_k(\lambda): \lambda \in \Lambda)$  is a
Markov chain in $k$. We can represent it as an $4m^2 |\Lambda|$ dimensional complex vector where the entries are labeled with $(i,j,\lambda)$.
In order to prove that it converges to
the appropriate SDE we use Proposition \ref{p_turboEK} of
Section \ref{section:Tools}.

We first rewrite $R_k=\frac{\sigma}{\sqrt{n}} R_k'+ \frac{\lambda}{n}R_k''$ where $R_k'$ contains
$V_k^O$ and $R_k''$ contains $ I$ from
the middle term of $R_k$. We first focus on the noise term
$R_k'$ which expands to
\[
  R_k'=  iS\mat{Z^k}00{Z^{-k}}\mat{I}{-Z}{-I}{\bar Z}\mat
0{V_k^O}00\mat{\bar Z }{Z}{I}{I} \mat{Z^{-k}}00{Z^{k}} S
\]
the three middle factors simplify and we get
\begin{equation}\label{evoR'}
  R_k'=iS\mat{Z^kV_k^OZ^{-k}}{Z^kV_k^OZ^k}{-Z^{-k}V_k^OZ^{-k}}{-Z^{-k}V_k^OZ^k}
 S\end{equation}
 %=iS\mat AB{-\bar B}{-\bar A}S,
 (the diagonal block entries are negative conjugates of
each other (note that $V_k^O$ is real) and so are the
off-diagonals). For the convergence to the limit, we need
to understand the covariance of the partial sums of $R_k$
over $k$. (This is needed for condition (\ref{e fia}) in
Proposition \ref{p_turboEK} .)

For this, we may ignore the $S$ factors for the moment, and
study
\begin{equation}\label{osc} A_\ell=\sum_{k=1}^\ell Z^kV_k^OZ^{-k}, \qquad
B_\ell=\sum_{k=1}^\ell Z^{k}V_k^OZ^k
\end{equation}
The $i,j$ entry of the first term is
$$
\sum_{k=1}^\ell (z_i\bar z_j)^k\sum_r
O^T_{j,r}v_{k,r}O_{r,i}
 = \sum_r \sum_{k=1}^\ell v_{k,r}  (z_i\bar z_j)^k O_{r,j}O_{r,i}
$$
Consider the complex covariance of the $i,j$ and $i',j'$
entries in $A_\ell$.  Since that matrix is Hermitian, we
may assume $i\le j$ and $i'\le j'$.  Since the $v$'s are
independent, the covariance is given by
$$
\ev A_{i,j} \bar A_{i',j'}=\sum_{k=1}^\ell  (z_i\bar
z_j\bar z_{i'}z_{j'})^k \sum_r
O_{r,j}O_{r,i}O_{r,j'}O_{r,i'} =
\begin{cases}
\ell \langle |O_i|^2, |O_j|^2  \rangle &\text{for~}(i,j)=(i',j')\,,\\
\ell \langle |O_i|^2, |O_{i'}|^2  \rangle &\textup{for }
(i,i')=(j,j')\,,
\\
\mathcal O(1)  &\text{otherwise} \,.
\end{cases}$$
This is because we can write $z_i\bar z_j\bar
z_{i'}z_{j'}=e^{i \alpha}$ with $|\alpha|\le \pi$ and
unless $\alpha=0$ the sum will be $\mathcal
O(|\alpha|^{-1})$. By the chaoticity condition $\alpha=0$
can happen only if the conjugated $z_i$ are matched up in
pairs to equal non-conjugated $z_i$, in which case the
indices have to be the same. The only options are $i=i',
j=j'$ or $i=j, i'=j' $.

This explains that the correlation between an entry  $A$
and the entry of $B$ is always $\mathcal O(1)$: there will
be an odd number of conjugated $z$'s so such a matching
will not occur.

The expectation of the product on the other hand is given
by
$$
\ev A_{i,j}A_{i',j'}=\sum_{k=1}^\ell  (z_i\bar
z_jz_{i'}\bar z_{j'})^k \sum_r
O_{r,j}O_{r,i}O_{r,j'}O_{r,i'} =
\begin{cases}
\ell \langle O_i^2, O_{i'}^2  \rangle &\text{for~} i=j \text{ and } i'=j'\,,\\
\mathcal O(1)  &\text{otherwise} \,.
\end{cases}$$
indeed, these are the only possible matchings for the
conjugated and not-conjugated $z$'s if $i\le j$ and $i'\le j'$.
Now
$$
\ev B_{i,j}B_{i',j'} =\mathcal O(1)
$$
since no matching can occur in the relevant product
$z_iz_jz_{i'}z_{j'}$. Finally, for $i\le j$ and $i'\le j'$
(since $B$ is a complex symmetric matrix)
$$
\ev B_{i,j}\bar B_{i',j'} =\begin{cases}
\ell \langle O_i^2, O_{j}^2  \rangle &\text{for~} i=i' \text{ and } j=j'\,,\\
\mathcal O(1)  &\text{otherwise} \,.
\end{cases}.$$
Turning to the drift term $R_k''$ we first simplify it to
get
\begin{equation}\label{evoR''}
R_k''=i  S
\mat{I}{Z^{2k}}{-Z^{-2k}}{-I} S.
\end{equation}
This leads to the estimate
\begin{equation}
i\sum_{k=1}^\ell S
\mat{I}{Z^{2k}}{-Z^{-2k}}{-I} S=i
S\mat{\ell I }{\mathcal O(1) }{\mathcal
O(1)}{-\ell I }
 S . \label{drift}
 \end{equation}

\begin{remark}
From (\ref{evoR'}) and (\ref{evoR''}) it is clear that
$R_k$ is of the form $\mat{a}{b}{\bar b}{\bar a}$ and thus
this will hold for $I+R_k$ as well. The product of such
matrices will have the same structure, which explains the
last assertion of Section \ref{section:reg}.

\end{remark}

We now turn back to the proof of Proposition \ref{letm}.
For the proof of the convergence, we will use Proposition
\ref{p_turboEK}, for which we need to verify the conditions
listed there. Note that the proposition deals with the convergence of $\R^d$ valued Markov chains and here we are dealing with vectors with entries as $2m\times 2m$ complex matrices. As we have already mentioned we can just treat our Markov chain as a $4m^2|\Lambda|$ dimensional complex vector, and by taking real and imaginary parts we may apply the proposition. Note that Proposition
\ref{p_turboEK} immediately extends to complex valued Markov chains as well, the only modification is that one has to introduce
\begin{align}\label{defat}
\tilde a^{n}(t,x):=\ev[ Y_{\lfloor nt\rfloor}^n(x) {\overline{ Y}_{\lfloor
nt\rfloor}^n}(x)^\textup{T}]
\end{align}
as well, and assume that  $\tilde a$ satisfies the same conditions as $a$. This is because the covariance structure of a complex vector $Y$ can be computed from the quantities $\ev (Y-\ev Y)(Y-\ev Y)^T$ and  $\ev (Y-\ev Y)(\bar Y-\ev \bar Y)^T$.

In our case the conditional distribution of $X_{k+1}-X_{k}$ given $X_{k}=x$ is given by
\begin{align*}
Y_k(x)= ( T_*^{(O\otimes
I_2)Q})^{-k-1}(T_{k+1}T_*^{-1}-I)^{(O\otimes
I_2)Q}(T_*^{(O\otimes I_2)Q})^{k+1} \; x.
\end{align*}
The functions $a^n(t,x)$, $\tilde a^n$ and $b^n(t,x)$ are  defined according to (\ref{defab}) and (\ref{defat}):
\begin{align*}
a^n(k/n,x)&=\ev
\left[(Y_{k}(x))_{i,j,\lambda}(Y_k(x))_{i',j',\lambda'}
\;\vert\;X^{(n)}_{k}=x\right]\\
\tilde a^n(k/n,x)&=\ev
\left[(Y_{k}(x))_{i,j,\lambda}(\bar Y_k(x))_{i',j',\lambda'}
\;\vert\;X^{(n)}_{k}=x\right]\\
b^n(k/n,x)&=\ev
\left[(Y_{k}(x))_{i,j,\lambda}
\;\vert\;X^{(n)}_{k}=x\right].
\end{align*}
The entries of $a, \tilde a$ are indexed by $(i,j,\lambda)$,  $(i',j',\lambda')$ and the entries of $b$ are indexed by $(i,j,\lambda)$.

\begin{itemize}
\item The boundedness of the cubic terms (\ref{e 3m}) is proved as
    follows. It is clear that  $Y_k(x)$ is a
    bounded linear function of $\lambda$'s and the $v_i$'s with
    coefficients given by the entries of $x$. So as long as $x$ is bounded, condition
   (\ref{e 3m})  holds because of the third moment assumption on the $v_i$'s.
   To ensure that
    $x$ is bounded, we first consider a truncated process in
    which $a^n, \tilde a^n$ and $b^n$ are multiplied by a smooth version of the cutoff function $\mathbf{1}(\|x\|\le c)$. This will basically stop $X_k^n$ once $\|X_k^n\|\ge c$.
     It will follow
     from the proposition that the truncated process
    converges to the truncated version of the limit, for every
    $c>0$. However, as $c\to\infty$, the solution of the
    truncated SDE is with high probability equal to the
    non-truncated one. This follows from the fact that the SDE (\ref{mainsde}) is linear and therefore it does not blow up in finite time.
    It follows that the truncated processes converge (\ref{mainsde}) from which we also get that the original processes must converge there as well.

\item Conditions (\ref{e lip}) and (\ref{e fia}) follow
    from the calculations in the first part of the proof.
We check the conditions for $a$ and $b$, for $\tilde a$ it will follow similarly.

Instead of working directly with $a$ and $b$, we introduce
\begin{align}\label{hatY}
\hat Y_k=\frac{\lambda}{n} \mat{I}{Z^{2k}}{-Z^{-2k}}{-I} +
 \frac{1}{\sqrt{n}}\mat{Z^kV_k^OZ^{-k}}{Z^kV_k^OZ^k}{-Z^{-k}V_k^OZ^{-k}}{-Z^{-k}V_k^OZ^k}
\end{align}
and let $\hat a^n(t)$ and  $\hat b^n(t)$ be the $n$ times
the second and first moments of the  $4m^2|\Lambda|$-vector
$\hat Y_{\lfloor nt \rfloor}$. Then we have
$$
X_{k}-x= iS \hat Y_kSx
$$
given $X_{k-1}=x$. This is a linear function of $x$, and
since $x$ is bounded by truncation, it suffices to check
(\ref{e fia})  for $\hat a$ and $\hat b$ rather than the
original $a,b$. Similarly,  (\ref{e lip}) for (the
truncated) $a,b$ is implied by
\begin{eqnarray*}
|\hat a^n(t)|+|\hat b^n(t)|&\le& c
\end{eqnarray*}
This, in turn, is follows from  the expression for $\hat
Y_k$ above (note that $Z$ is diagonal with unit
complex numbers).

Returning to (\ref{e fia}), we first check the existence
of $\hat b$ so that
\begin{eqnarray*}
\sup_{t} \Big|\int_0^t \hat b^n(s)\,ds-\int_0^t \hat
b(s)\,ds \Big| &\to& 0.
\end{eqnarray*}
This is clear by the following computation, based on
(\ref{drift}):
$$ \hat b(k/n)=
\sum_{\ell=1}^k \ev \hat Y^{(n)}_\ell = \sum_{\ell=1}^k
\frac{\lambda}{n} \mat{Z^kI
Z^{-k}}{Z^kIZ^k}{-Z^{-k}IZ^{-k}}{-Z^{-k}IZ^k} = \lambda
\frac{k}{n} \mat{I}{0}{0}{-I} + \mathcal O(n^{-1})
$$
where the $\mathcal O$ is uniform in $k$.

To check the
second moment terms,   we will consider the
the $\lambda=\lambda'$ case, the general case is similar.
We first note that since the mean
increments are of $\mathcal O(1/n)$, it suffices to look at
the variance  instead of the second moment. Thus we need the covariance matrix of (the vector version of)
$$
\frac{1}{\sqrt{n}}\mat{Z^kV_k^OZ^{-k}}{Z^kV_k^OZ^k}{-Z^{-k}V_k^OZ^{-k}}{-Z^{-k}V_k^OZ^k}.
$$
This, by the computations above (starting at equation
(\ref{evoR'})), converges to the covariance matrix of the
corresponding noise term in the SDE.
\end{itemize}

For the statement about the convergence of the analytic functions $X_k^{(n)}(\lambda)$, we will need a bound of the
form
\begin{equation}\label{uniform bound}
\ev \|X^{(n)}_k(\lambda)\| \le f(\lambda,k/n)
\end{equation}
where $f(\lambda,t)$ is a locally bounded function. Then by Corollary \ref{analytic}
the  claim  follows.

We have
\begin{eqnarray*}
\ev (\|X_k\|^2-\|X_{k-1}\|^2\,|\,\mathcal F_k) &=& \ev \tr
(\Delta X X_{k-1}^T+\Delta X^T X_{k-1}+\Delta X \Delta X^T) \\&\le &
\frac{c(1+\lambda^2)}{n} \|X_{k-1}\|^2
\end{eqnarray*}
Taking expectations we now see that
$$\ev \|X_k\|^2 -\ev
\|X_{k-1}\|^2 \le \frac{c(1+\lambda^2)}{n} \ev
\|X_{k-1}\|^2
$$
rearranging and taking a product for $1,\ldots, k$ we
immediately get the estimate
$$\ev \|X_k\|^2 \le \left( 1+\frac{c(1+\lambda^2)}{n}\right)^k
\|X_0\|^2\le \exp(c(1+\lambda^2)\frac{k}{n})\|X_0\|^2
$$
from which the bound \eqref{uniform bound} follows.
\end{proof}

\begin{proof}[Proof of Corollary \ref{eigenvalues
converge}]

By Skorokhod embedding \cite{Kallenberg}, we can realize the distributional
convergence of the matrix-valued  analytic functions
$X^{(n)}(\cdot)\Rightarrow Y_1(\cdot)$ as almost sure convergence. In
particular, along our subsequence, we have a.s.
\begin{eqnarray*}
\det ((\bar Z^{n+1},Z^{n+1})X^{(n)}_n(I,-I)^T) \rightarrow
\det ((\bar Z_*,Z_*)Y_1(\lambda)(I,-I)^T)
\end{eqnarray*}
uniformly on compacts, and the limit is analytic in
$\lambda$. Note that for $\lambda=0$ the matrix in the
determinant equals
\begin{eqnarray*}
2i \Im(\bar Z_*(Y_1(0)_{11}-Y_1(0)_{12}).
\end{eqnarray*}
whose distribution is absolutely continuous with respect to
the distribution of the GOE; this follows from the SDE
\eqref{mainsde}. Thus the determinant as a function of $\lambda$  is not identically zero with
probability 1, and the zeros of the subsequence converge to
the zeros of the limit almost surely in our realization.
The distributional convergence follows.
\end{proof}

In order to prove Theorem \ref{thm2} we will also need the following lemma:

\begin{lemma}\label{l:aux}
Fix a $G$ with distinct eigenvalues $d_1, \dots, d_m$  with $\max_j |d_j|=\tilde d$ and fix $0<r<2/\tilde d$.
 Then for a.e. $\lambda_*\in(-(2-r \tilde d),2-r \tilde d)$ there is a sequence $n_\nu$ so that the points
$q_j=\arccos((\lambda_*-r d_j)/2)$ are chaotic and we also have $(n_\nu+1) q_j \to 0 \mod 2\pi$ for all $1\le j\le m$.
\end{lemma}

\begin{proof}
If $\lambda_*\in(-(2-r \tilde d),2-r \tilde d)$  then the vector $q=q(\lambda_*)=(\arccos
((\lambda_*-r d_j)/2)_{j=1,\ldots,m}$ is well-defined because of our conditions.
We will show that for a.e.~$\lambda^*$ in that interval the vector $q$ has no nonzero
integer vector orthogonal to it. Then the set $\{q_j\}_{j=1,\ldots,m}$ is chaotic and because
 the orbit
$\{nq  \mod 2\pi \;:\: n\ge 0\}$ is dense on the $m$-torus
 we can find a subsequence $n_\nu$ so that $n_\nu q_j$  converges to $0$ mod $2\pi$ for all $j$.

 It suffices to show that for any fixed nonzero integer
vector $w$, only finitely many $\lambda$ has $q(\lambda) \cdot
w=0$. Note that $\lambda\mapsto q(\lambda)\cdot w$ is an analytic
function on the interval $(-(2-r\tilde d),2-r \tilde d)$ so it has finitely many zeros there or it has to be constant zero. We will show that the latter is not possible which will finish the proof.

The function $$q_j'(\lambda)=-\frac{1}{2 \sqrt{1-\frac{1}{4} (\lambda -r d_j)^2}}$$ has singularities at
$\lambda^\pm_j=\pm 2-r d_j$. %where $\lambda_j^-<0$ and  $\lambda_j^+>0$.
If $q(\lambda) \cdot w=0$ for all values of $\lambda$ then these singularities must cancel out. But this is  impossible if $w\neq \underline{0}$ as the numbers $\lambda^-_j, \lambda_j^{+}, j=1,\dots, m$ are all different:  $-2-rd_j=2-r d_k$ would imply $4=r (d_j-d_k)\le 2 r \tilde d<4$.
\end{proof}

\begin{proof}[Proof of Theorem \ref{thm2}]
Fix $r$ and $m$ and consider $G=Z_m$. This has distinct eigenvalues $d_j$ with $|d_j|\le 2 \cos(\pi/(m+1))$. Then we may apply  Lemma \ref{l:aux} and for almost every $\lambda^*\in (-2(1-r\cos(\pi/(m+1))),2(1-r\cos(\pi/(m+1)))$
there is a sequence $n_\nu$ with the properties given by the lemma.  Next we may apply Proposition \ref{letm} and Corollary \ref{eigenvalues converge} with $ Z_*=I$ to show that the  eigenvalues of (\ref{scaled spec}) converge to 
 $$
 \Lambda_\sigma:=\zeros_\lambda \left[
 \det ((I,I)Y_1(\lambda)(I,-I)^T)\right]
 $$ 
where $Y_t(\lambda)$ is the solution of the SDE (\ref{mainsde}). This proves the first and the third statements. For the convergence of $\sigma^{-1} \Lambda_\sigma$ we first note that with the notation $\hat Y_t(\lambda)=\sigma^{-1} ( Y_t(\sigma \lambda)-I)$ we have
\begin{align*}
\sigma^{-1} \zeros_\lambda \left[ \det ((I,I)Y_1(\lambda)(I,-I)^T)\right]
\ed
 \zeros_\lambda \left[\ \det ((I,I)\hat Y_1(\lambda)(I,-I)^T) 
\right]
\end{align*}
 From (\ref{mainsde}) it follows that $\hat Y$ satisfies
 \begin{equation}\label{approxde}
d\hat Y = S^2 \mat I00{-I} i \lambda dt (\sigma \hat Y+I) + iS\mat
{dA}{dB}{-d\bar B }{-d\bar A}S(\sigma \hat Y+I), \qquad \hat Y_0=0.
\end{equation}
As $\sigma\to 0$ the  solution of this SDE converges to the solution of the SDE with $\sigma=0$ (see e.g.~Theorem 11.1.4 of \cite{SV}) which is just a
 matrix valued Brownian motion with drift. In particular
\begin{align*}
\det ((I,I)\hat Y_1(\lambda)(I,-I)^T) &\to 
\det \left[(I,I)
iS\mat
{A(1)+\lambda}{B(1)}{-\bar B(1) }{-\bar A(1)-\lambda}S(I,-I)^T\right]\\
&=\det \left[2iS_{11}(\lambda I+\Re A(1) -\Re B(1))S_{11}\right],
\end{align*}
the zeros of which are $\spec(\Re B(1)-\Re A(1))$. Here $A(1)$, $B(1)$ are Gaussian matrices described in Remark \ref{remG}.
 Note that $\Re B(1)-\Re A(1)$
can be written as $(m+1)^{-1/2} (K+bI)$, where $b$ is a
standard normal random variable, and $K$ is a  a real symmetric matrix with independent mean zero
real normal entries so that
\begin{equation}
\ev K_{ij}^2=\begin{cases}1 & i\not=j \\5/4 &i=j
\end{cases}.\label{almostGOE}
\end{equation}
The $bI$ term amounts to
a random shift of the local eigenvalue process.
\end{proof}

\section{GOE as a limit}\label{section:GOE}

In order to get the GOE limit, we need to change the standard deviation $\sigma$ with $n$.

\begin{proposition}\label{shorttime}
Let $G$ be a weighted graph  with distinct eigenvalues
$d_i$, and fix  $0<r<2/\max_j |\,d_j|.$ Suppose that
\begin{itemize}
\item the conclusion of Lemma \ref{l:aux} holds for  $\lambda_*$, i.e.~the critical angles
    $q_j=\arccos ((\lambda_*-r d_j)/2)$   exist, are chaotic and
there is a sequence $n_\nu$ with $(n_\nu
    +1)q_{j}\to 0 {\mod 2 \pi}$,

\item $\sigma_\nu$ is a sequence with
  $\sigma_\nu\to 0$ and  $\frac1{\sigma_\nu} \max_j \left|e^{i(n_\nu+1) q_{j}}-1\right|\to
    0$,

\item we have a sequence $V_\nu$ of diagonal perturbation matrices
    where the entries are independent, have mean 0, variance 1
    and uniformly bounded third moment.

\end{itemize}

Then the regularized transfer matrix evolution $X^\nu_k$ for the
operator
 \begin{equation}\label{Mnu}
M_\nu=\frac{n_\nu}{\sigma_\nu}\left(rG\times \mathbb
Z_{n_\nu}+ \frac{\sigma_\nu}{\sqrt{n_\nu}}V_\nu-\lambda_*I
\right)
\end{equation}
 satisfies the following. For any finite  $\Lambda
\subset \mathbb C$ we have convergence in distribution for
the regularized transfer matrices
$$
(\frac{1}{\sigma_\nu}(X^{\nu}_{\lfloor n_\nu
t\rfloor}(\lambda^*+\lambda\sigma_\nu/
n_\nu)-I),\;t\in[0,1],\lambda\in \Lambda)\Rightarrow
(Y_t(\lambda),\;t\in[0,1],\lambda\in \Lambda)
$$
where
$$
Y = iS^2 \mat I00{-I} \lambda t + iS\mat {A}{B}{-\bar
B }{-\bar A}S.
$$
Here $A$,$B$ are matrix-valued
Brownian motions with covariance structure given by
\eqref{covariances1}, \eqref{covariances2}.

Moreover, the eigenvalue process  of $M_\nu$ converges  to
the eigenvalues of $$\Re(B(1)-A(1)).$$
\end{proposition}

\begin{remark}\label{GOE} For the $G=\Z_m$ case the proposition shows that instead of letting $n\to \infty$ first and then $\sigma\to 0$ to get the eigenvalues of a Gaussian matrix as the limit of the spectrum (as in Theorem \ref{thm2}) we can achieve this by changing $\sigma$ with $n$.

In the case when the graph is $\Z_m$,  the matrix $\Re(B(1)-A(1))$ is the one we got in the proof of  Theorem \ref{thm2},
it can be written as $(m+1)^{-1/2} (K+bI)$, where $b$ is a
standard normal random variable, and $K$ is a version of
the GOE: a real symmetric matrix with independent mean zero
real normal entries with covariance structure (\ref{almostGOE}).
Note that for the GOE, the diagonal terms have variance 2. The
distribution of the matrix $K$ is not invariant under
orthogonal conjugation. The $bI$ term amounts to
a random shift of the local eigenvalue process.
\end{remark}

\begin{proof}
Let $\tilde X^\nu_k(\lambda)=(\frac{1}{\sigma_\nu}(X^{\nu}_{\lfloor n_\nu
t\rfloor}(\lambda^*+\lambda\sigma_\nu/
n_\nu)-I)$.
From (\ref{evoR}) (and the following discussion) we get
\begin{align*}
\tilde X_k-\tilde X_{k-1}=\sigma^{-1} R_k(\sigma \lambda)(\sigma \tilde X_k+I)= \left( \frac{1}{\sqrt{n}}R_k'+\frac{\lambda}{n} R_k''  \right)(\sigma \tilde X_k+I)
\end{align*}
 The proof of the SDE convergence of $\tilde X$ follows very closely the proof of
Proposition \ref{letm}. The only difference is that because of $\sigma\to 0$ the functions $a(t,x), \tilde a(t,x), b(t,x)$ will not depend on $x$, that is why we get the matrix valued Brownian motion in the limit.

For the convergence of eigenvalues, we write %\benedek{kiradirozni  az $m$-eket, esetleg egg megjegyzes a dimenziokrol}
\begin{align*}
\frac{1}{\sigma_\nu}(\bar Z^{n+1},Z^{n+1})(X_{n_\nu}^{(n)}-I)(I,-I)^T&=\frac{1}{\sigma_\nu}(\bar
Z^{n+1}-I,Z^{n+1}-I)X_{n_\nu}^{(n)}(I,-I)^T\\
&\hskip-30pt+\frac{1}{\sigma_\nu}(I,I).(I,-I)^T+
\frac{1}{\sigma_\nu}(I,I)(X_{n_\nu}^{(n)}-I)(I,-I)^T.
\end{align*}
Note that $I$ is the $m\times m$ identity matrix, except  in $X_{n_\nu}^{(n)}-I$ where it is $2m$-dimensional.
%
%
%$$\frac{1}{\sigma_\nu}(\bar Z^{n+1},Z^{n+1})X_{n_\nu}^{(n)}(I,-I)^T
%$$
%we write the above as
%$$
%\frac{1}{\sigma_\nu}(\bar
%Z^{n+1}-I,Z^{n+1}-I)X_{n_\nu}^{(n)}(I,-I)^T+\frac{1}{\sigma_\nu}(I,I).(I,-I)^T+
%\frac{1}{\sigma_\nu}(I,I)(X_{n_\nu}^{(n)}-I)(I,-I)^T
%$$
The first term on the right converges to $0$ since
$X_{n_\nu}^{(n)}(I,-I)^T$ is tight and
$(Z^{n+1}-I)/\sigma_\nu\to 0$ by assumption. The second
term vanishes, so considering the third term we conclude that
$$
%\frac{1}{\sigma_\nu}(\bar
%Z^{n+1},Z^{n+1})X_{n_\nu}^{(n)}(I,-I)^T
\frac{1}{\sigma_\nu}(\bar Z^{n+1},Z^{n+1})(X_{n_\nu}^{(n)}-I)(I,-I)^T
\rightarrow
(I,I)Y_\lambda(1) (I,-I)^T.
$$
 uniformly  on
compacts. For a real $\lambda$ this is equal to $2S_{11}(\lambda I+\Re(A(1)-B(1)))S_{11}$ and its determinant  will be zero exactly at the eigenvalues of $\Re(B(1)-A(1))$. From this the second part of the proposition follows by the same arguments as in the proof of Corollary \ref{eigenvalues
converge}.
%We now take determinants and use the fact that
%convergence of analytic functions implies the convergence of
%their zeros, and the claim follows.
\end{proof}

\begin{corollary} \label{Sine1 limit} Fix $0<r<1$. Then for almost all $\lambda_*\in [-2(1-r),2(1-r)]$
 there exists a sequences of integers $m_\nu\to \infty, n_\nu\to \infty$, positive numbers $\tilde \sigma=\tilde \sigma_\nu\to 0$, noise matrices $V=V_\nu $ and coefficients $\gamma=\gamma_\nu$ so that
the eigenvalue process of
\begin{align}\label{scaled}
\gamma\left[(r \mathbb Z_{m})\times  \mathbb Z_{n}+\tilde \sigma V-\lambda_*I\right]
\end{align}
as $\nu\to\infty$ converges locally to the Sine$_1$ process. Here
the noise matrices $V $ are diagonal, the
 entries are independent, have mean 0, variance 1
    and uniformly bounded third moment.

Moreover, if $r=1$ then we may choose a sequence $\lambda_*=\lambda_*^{(\nu)}$ so that (\ref{scaled}) converges to the Sine$_1$ process.
\end{corollary}
Corollary \ref{Sine1 limit} immediately proves Theorem \ref{thmmain}.
\begin{proof}
First suppose that $0<r<1$. Note that the eigenvalues $d_j$  of $\Z_m$ are distinct and $\max |d_j|=2 \cos(\pi/(m+1))<2$.

We  consider a $\lambda_*$ so that the statement of Lemma \ref{l:aux} holds for all values of $m$. We start with a fixed $m$  and consider the appropriate sequence $n_\nu$.
Since $(n_{\nu}+1) q_j\to 0 \mod 2\pi$ we can choose  $\sigma_\nu=\sigma_{\nu,m}\to 0$ so that we also have $\frac1{\sigma_\nu} \max_j \left|e^{i(n_\nu+1) q_{j}}-1\right|\to 0$. Then we may apply Proposition
\ref{shorttime} with $G=\mathbb Z_{m}$  which means that the eigenvalues of \eqref{Mnu}
$M_\nu$ converge to those of $(m+1)^{-1/2}(K+bI)$ where $K, b$ are described in Remark \ref{GOE}.

As $m\to \infty$, by the methods of \cite{EY} the bulk
eigenvalue process of the matrices $\sqrt{m}(K+bI)$ (see Remark
\ref{GOE}) converges locally to the Sine$_1$ process.
The proof of this statement will be expanded in Lemma \ref{l:locallimit} in the Appendix.

This means that with $\tilde \sigma_\nu=\sigma_\nu/\sqrt{n}$ and an appropriate $\gamma_\nu$ the eigenvalue process (here denoted $s_{m,n}$) of (\ref{scaled}) converges to  the spectrum $s_{m}$ of $\sqrt{m}(K+bI)$, and as $m\to \infty$ this will converge to  the Sine$_1$ process. In short, in the topology of weak convergence, we have
$$\begin{array}{lcll}
s_{m,n} &\to& s_m , \qquad &n\to \infty,\ m \mbox{ fixed,}\\
s_{m} &\to& \text{Sine}_1, & m \to\infty.
\end{array}
$$
The standard diagonalization argument provides a sequence $n_m\to\infty $ so that $s_{m,n_m}\to$ Sine$_1$.

For the $r=1$ case note that for a fixed $m$ by Lemma \ref{l:aux} we can find a $\lambda_*^{(m)}\in (-4 \sin^2(\pi/(2+2m)),4 \sin^2(\pi/(2+2m))$ for which Proposition \ref{shorttime} can be applied. Thus for each $m$ (with an appropriate $\lambda_*^{(m)}$ centering) the spectrum of the rescaled process (\ref{scaled}) converges to that of $\sqrt{m}(K+bI)$ along an appropriate subsequence. The same diagonalization argument as in the $r<1$ case produces a sequence along which we have convergence to the Sine$_1$ process.
\end{proof}
\begin{remark}[Higher dimensional boxes]\label{r:hd}
When $G=\Z_{m_1}\times \ldots \times \Z_{m_d}$ then the eigenvectors of $G$
are of the form
\[
O_{j_1, \dots, j_d}(k_1,\dots,k_d)=\prod_{\ell=1}^d \sqrt{\frac{2}{m_\ell+1}} \sin\left(\frac{\pi j_\ell k_\ell}{m_\ell+1}\right)
\]
and the corresponding eigenvalues are $d_{j_1,\dots, j_d}=2 \sum_{\ell=1}^d  \cos\left(\frac{j_\ell \pi }{1+m_\ell}\right)$.
One can also calculate that
\begin{equation}\left(\prod_{\ell=1}^d (m_\ell+1)\right)\langle |O_{\underline{i}}|^2, |O_{\underline{j}}|^2 \rangle=
\prod_{\ell=1}^d (1+\frac12 1(i_\ell=j_\ell))\label{isotropic2}.\nonumber
\end{equation}
If the eigenvalues are distinct then  Lemma \ref{l:aux} and Proposition \ref{shorttime} will still apply. However, the limiting symmetric real Gaussian matrix $\Re\left(B(1)-A(1)\right)$ will have a more complicated covariance structure for $d>1$ then the one described in Remark \ref{GOE} and the current results are not strong enough to imply that the bulk scaling limit will be the Sine$_1$ process.

One can check that in that case $\Re\left(B(1)-A(1)\right)=M_1+M_2$ where $M_1$ is a constant times a GOE matrix and $M_2$ is an independent Gaussian matrix with a non-trivial covariance structure. Because of the component $M_1$ the local relaxation flow  argument of \cite{EY} will go through. The problem is caused by the fact that for the other component $M_2$ the strong local semicircle result is not available.
\end{remark}

\section{Appendix}
\label{section:Tools}

\subsection{SDE limit of Markov chains}

The following is the main tool for proving convergence in
the osciallatory setting.  It is Proposition 23 in
\cite{BVBV}, and is based on Theorem 7.4.1 of
\cite{EthierKurtz}.

\begin{proposition}\label{p_turboEK}
Fix $T>0$, and for each $n\ge 1$ consider a Markov chain
$$
(X^n_\ell\in \mathbb R^d,\, \ell =1\ldots  \lfloor nT
\rfloor).
$$
Let $Y^n_\ell(x)$ be distributed as the increment
$X^n_{\ell+1}-x$ given $X^n_\ell=x$. We define
\begin{align}\label{defab}
b^n(t,x)= n \ev[ Y_{\lfloor nt\rfloor}^n(x)],\qquad
a^n(t,x)=n\ev[ Y_{\lfloor nt\rfloor}^n(x) Y_{\lfloor
nt\rfloor}^n(x)^\textup{T}].
\end{align}
Suppose that as $n\to \infty $ we have
\begin{align}
|a^n(t,x)-a^n(t,y)|+|b^n(t,x)-b^n(t,y)|&\le
c|x-y|+o(1)\label{e lip}\\
 \sup_{x,\ell}  \ev[|Y^n_\ell(x)|^3] &\le cn^{-3/2} \label{e 3m},
\end{align}
and that there are functions $a,b$ from $\R\times[0,T]$ to
$\R^{d^2}, \R^d$ respectively with bounded first and second
derivatives so that
\begin{align}
\sup_{x,t} \Big|\int_0^t a^n(s,x)\,ds-\int_0^t a(s,x)\,ds
\Big|+\sup_{x,t} \Big|\int_0^t b^n(s,x)\,ds-\int_0^t
b(s,x)\,ds \Big| &\to 0. \label{e fia}
\end{align}
Assume also that the initial conditions converge weakly:
$$
X_0^n\Rightarrow X_0.
$$
Then $(X^n_{\lfloor n t\rfloor}, 0 \le t \le T)$ converges
in law to the unique solution of the SDE
$$
dX = b \,dt + g\, dB, \qquad X(0)=X_0$$
where $B$ is a $d$-dimensional standard Brownian motion and $g$ is any  $C^2$ function satisfying
$gg^{\rm T}=a.$
\end{proposition}

We need a result that strengthens the convergence of random
analytic functions.

\begin{proposition}\label{Montel} Let  $f_n$ be a sequence of random analytic functions on an open set $D\subset \CC$ so that
$\max_A |f_n|$ is tight for every closed ball $A\subset D$ and
$f_n\Rightarrow f$ in the sense of finite dimensional
distributions. Then $f$ has a unique analytic version and
$f_n \Rightarrow f$ in distribution with respect to
local-uniform convergence.
\end{proposition}

\begin{proof}
Pick a countable dense set of points $D'\subset D$ and  let $\mathcal A$ be the countable  set  of closed balls $A\subset D$ with center in $D'$ and rational radius.
We can
first find a subsequence so that the joint distribution of
$\max_A |f_n|$, $A\in \mathcal A$ and  $f_n(z), z\in D'$ converges. By the
Skorokhod embedding theorem \cite{Kallenberg} we can realize the sequence
$(f_n(z),z\in D')$ on a single probability space $\Omega$
so that almost surely $f_n(z)\to f(z)$ for all $z \in D'$
and also $\max_A |f_n|$ converges (and so it is bounded)
for all $A\in \mathcal A$. By continuity, we can define the
$f_n$ as analytic functions on $D$ for the probability
space $\Omega$. Then a.s.~the following holds: for every compact set $B$ there exists a random constant $c_B$ with $\max_{z\in B, n} |f_n(z)|\le C_B$.
To see this, note that $B$ is covered by $ \cup_{A\in \mathcal A}\textup{interior}(A)$ and since $B$ is compact we can choose a finite sub-cover.
Then the sequence $f_n$
has at least one analytic
limit in the sense of uniform-on-compacts convergence by
Montel's theorem. This limit must agree with $f$ on all
points $z\in D'$, so it is unique.

By the above argument any sequence has a further
subsequence that converges locally uniformly to some
analytic $f$ in distribution. But the distribution of $f$
is determined by its finite dimensional distributions, so
the limit is unique.
\end{proof}

\begin{corollary}\label{analytic}
Let $f_n$ be random analytic functions on an open set $D\subset \CC$ so that $\ev
|f_n(z)| \le g(z)$ for a locally bounded function $g$. Assume that
$f_n(z)\Rightarrow f(z)$ in the sense of finite dimensional
distributions. Then $f$ has a unique analytic version and
$f_n \Rightarrow f$ in distribution with respect to
local-uniform convergence.
\end{corollary}

\begin{proof}
Suppose that a closed disk $A$ is contained in a slightly bigger disk $B$ which is still contained in $D$. Then by Cauchy's integral theorem $\max_A |f_n|\le C \int_{\partial B} |f_n(z)| $
and by the condition on $\ev |f_n(z)|$ we get that $\ev \max_A|f_n|$ is uniformly bounded.
This implies that  $\max_A |f_n|$ is tight for every
closed ball $A\subset D$, and the claim follows from the
proposition.
\end{proof}

\subsection{The point process limit of the modified GOE}

Consider the $n\times n$ symmetric random matrix $H_n=n^{-1/2} (K+b I)$ where $b$ is a
standard normal random variable, and $K$ is the version of
the GOE considered in (\ref{almostGOE}).
We will show that the eigenvalue process of $H_n$ in the bulk converges  locally to the Sine$_1$ process, the bulk scaling limit of the GOE eigenvalue process. Bulk scaling means that we consider $\rho(\lambda) \sqrt{n} (H_n-\lambda I) $ where $|\lambda|<2$ and $\rho(\lambda)=\frac{1}{2\pi}\sqrt{4-\lambda^2}\, 1_{|\lambda|\le 2}$ is the semicircle density.

\begin{lemma}\label{l:locallimit}
For any $|\lambda|<2$ and  compactly supported continuous test function $ \Theta :\R^k\to \R$ we have
\begin{equation}
\int_{\R^n} d\alpha_1\dots d\alpha_n  \Theta (\alpha_1,\dots, \alpha_k) \rho(\lambda)^{-k}
(p_{H,n}^{(k)}-p_{GOE, n}^{(k)})\left(\lambda+\frac{\alpha_1}{n \rho(\lambda)}, \dots \lambda+\frac{\alpha_k}{\rho(\lambda) n}  \right)\to 0.\label{limit}
\end{equation}
Here $p_{H,n}^{(k)}$ is the $k$-point intensity function of the eigenvalues of $H_n$. The function  $p_{GOE, n}^{(k)}$ is the same for the $n\times n$ GOE matrix with variances $2/n$ and $1/n$ on the diagonal and elsewhere, respectively.
\end{lemma}
\begin{proof}
Since the proof works exactly the same way for any  $|\lambda|<2$, we will assume $\lambda=0$.

Our proof relies on the local relaxation flow arguments of \cite{EY} and the strong local semicircle law proved in \cite{EYY}. The argument in a nutshell is the following: if we have a real symmetric random matrix whose eigenvalues are well-approximated by the semicircle law locally then by adding a small constant times an independent GOE matrix (i.e. by running Dyson's Brownian motion for a small time with our matrix as the initial condition) the scaled eigenvalue process of the resulting matrix will be close to the Sine$_1$ process.

We will use  Theorem 2.3 of \cite{EYY} (see also the comments after the theorem) which provides a powerful quantitative version of the previous argument. Since we can write $\frac1{\sqrt{n}}K=K_1+K_2$  where $K_1$ is $1/2$ times a GOE and $K_2$ is an independent symmetric random matrix satisfying the conditions of the theorem then
 for any $q>0$ we have
\begin{align}\label{aux1}
&\big| \int_{-\frac{q}{\sqrt{n}}}^{\frac{q}{\sqrt{n}}}\frac{d\mu \sqrt{n}}{2q} \int_{\R^k} d\alpha_1\dots d\alpha_n \Theta (\alpha_1,\dots, \alpha_k) \rho(0)^{-k}\\
&\hskip10mm \times\nonumber
(p_{K,n}^{(k)}-p_{GOE, n}^{(k)})\left(\mu+\frac{\alpha_1}{n \rho(0)}, \dots, \mu+\frac{\alpha_k}{n \rho(0)}  \right)\big|\le C n^{-1/4+\eps} (q^{-1}+q^{-1/2}).
\end{align}
(We apply the theorem with $E=0, b=q/\sqrt{n}, t=1/2, \eps'\ll 1$ and $\delta=1-\eps'$.)
Since the spectrum of $H_n$ can be obtained by a random shift of the spectrum of $n^{-1/2} K$ we have
\begin{align}\label{aux2}
p_{H,n}^{(k)}\left(\frac{\alpha_1}{n \rho(0)}, \dots, \frac{\alpha_k}{n \rho(0)}  \right)=
&\int_{-\infty}^\infty \frac{db}{\sqrt{2\pi n^{-1}}} e^{-\frac{b^2 n}{2}}
p_{K,n}^{(k)}\left(b+\frac{\alpha_1}{n \rho(0)}, \dots, b+\frac{\alpha_k}{n \rho(0)}  \right).
\end{align}
By Fubini for any nonnegative function $F$ we have
\[
\int_{-\infty}^\infty \frac{db}{\sqrt{2\pi n^{-1}}} e^{-\frac{b^2 n}{2}} F(b)=\int_0^\infty e^{-\frac{q^2}{2}} \sqrt{\frac{2}{\pi }} q^2 \int_{-\frac{q}{\sqrt{n}}}^{\frac{q}{\sqrt{n}}}\frac{db \sqrt{n}}{2q} F(b)dq.
\]
Using this with (\ref{aux1}) and (\ref{aux2}) we get the upper bound
\[
\left|\textup{left hand side of (\ref{limit})} \right|\le\int_0^\infty e^{-\frac{q^2}{2}} \sqrt{\frac{2}{\pi }} q^2  C n^{-1/4+\eps} (q^{-1}+q^{-1/2})\le  C' n^{-1/4+\eps}
\]
which shows that the limit of the eigenvalue process of $H_n$ is the same as the limit of  $GOE$ eigenvalues multiplied by $n$ and shifted by an independent Gaussian $b_n$ of variance $n$. We first condition on the sequence of these independent Gaussians, so the GOE eigenvalues are now centered at $-b_n$, which, on the semicircle scale, converges to $0$.
In the classical literature, convergence is usually proved around a fixed window at $c$, $|c|<2$ (on the semicircle scale). The moving window case is rigorously proved in \cite{BVBV}. The lemma follows after removing the conditioning.
\end{proof}

\smallskip

\noindent{\bf Acknowledgements.} We thank Eugene
Kritchevski, Alex Bloemendal, and  L\'aszl\'o Erd\H os for useful discussions, and Hermann Schulz-Baldes for references.
We also thank the anonymous referees for valuable comments and suggestions.
The second author
thanks the hospitality of the Banff International Research
Station, where part of this work was conceived. This research is supported
by
the NSF Grant DMS-09-05820 (Valk\'o) and the NSERC discovery accelerator grant program and the Canada
Research Chair program (Vir\'ag).
%\bibliography{sse}

\def\cprime{$'$}

%\sc \bigskip \noindent B\'alint Vir\'ag, Departments
%of Mathematics and Statistics, University of Toronto,
%ON, M5S 3G3, Canada. \\{\tt balint@math.toronto.edu},
%{\tt www.math.toronto.edu/\~{}balint}

\end{document}